\documentclass[10pt]{amsart}
  \input amssymb.sty
\usepackage{eepic}

\usepackage{epsfig,latexsym,amsfonts,amssymb,amsmath,amscd,graphics,epic}
\usepackage{amsfonts,amssymb,amsmath,amscd,amsthm}
\usepackage[mathscr]{eucal}
\usepackage{mathrsfs}

\usepackage{mathbbol}

\usepackage{oldgerm,units}
\usepackage{wrapfig,epsfig}

\usepackage{ifthen}
\usepackage{mathbbol}

\usepackage{amsthm}
\usepackage[mathscr]{eucal}
\usepackage{mathrsfs}
\usepackage{mathbbol}
\usepackage{oldgerm,units}
\usepackage{wrapfig}

\newtheorem{theorem}{Theorem}[section]

\newtheorem{definition}[theorem]{Definition}

\newtheorem{claim}[theorem]{Claim}
\newtheorem{lemma}[theorem]{Lemma}

\newtheorem{corollary}[theorem]{Corollary}

\newtheorem{example}[theorem]{Example}
\newtheorem{remark}[theorem]{Remark}

\newtheorem{note}[theorem]{Note}

\newcommand{\Ref}[1]{(\ref{#1})}

\newcommand{\Real}{\mathbb R}
\newcommand{\Rati}{\mathbb Q}
\newcommand{\Int}{\mathbb Z}
\newcommand{\Net}{\mathbb N}

\newcommand{\chos}[2]{{#1 \choose #2}}

\newcommand{\Trop}{\mathbb T}

\newcommand{\trop}[1]{\mathcal{#1}}

\newcommand{\tI}{\trop{I}}

\newcommand{\tP}{\trop{P}}

\newcommand{\tS}{\trop{S}}

\newcommand{\tpF}{f}

\newcommand{\tfF}{\tilde{f}}

\newcommand{\To}{\longrightarrow }

\newcommand{\Dir}{\thSpc \Longrightarrow \thSpc}
\newcommand{\IFF}{\thSpc \Longleftrightarrow \thSpc}

\newcommand{\tUniS}{-\infty}

\newcommand{\al}{\alpha}
\newcommand{\bt}{\beta}

\newcommand{\sig}{\sigma}

\newcommand{\dl}{\delta}
\newcommand{\Dl}{\Delta}
\newcommand{\lm}{\lambda}
\newcommand{\Lm}{\Lambda}

\newcommand{\Var}{V}

\newcommand{\TrS}{\oplus}
\newcommand{\TrP}{\odot}

\newcommand{\OP}{\left(}
\newcommand{\CP}{\right)}

    \ifx\proof\undefined
    \newenvironment{proof}{
    \smallskip
    \noindent\emph{Proof.}}{\hfill\(\Box\)
    \bigskip
    } \fi

\newcommand{\vvMat}[9]{\small{\OP \begin{array}{ccc}
  #1 & #2 & #3\\
  #4 & #5 & #6\\
  #7 & #8 & #9\\
\end{array}\CP}}

\newcommand{\ifdef}[3]{\ifthenelse{\equal{#1}{true}}{#2}{#3}}

\newcommand{\thSpc}{\; \; \;}

\def\pSkip{\vskip 1.5mm \noindent}

\newcommand{\etype}[1]{\renewcommand{\labelenumi}{(#1{enumi})}}
\def\eroman{\etype{\roman}}

\def\({\left(}
\def\){\right)}

\def\dual{*}
\newcommand{\fdual}[1]{{#1}^{\dual}}
\def\dTrop{\fdual{\Trop}}
\def\tdTrop{\fdual{\tTrop}}

\def\eqR{\overset{_e}{\sim}}
\def\neqR{\overset{_e}{\nsim}}

\def\uf{f^i}

\def\la{\lm}
\def\sg{\sigma}

\def\completion{completion}
\def\supplement{supplement}
\def\Supplement{Supplement}

\def\reversal{reversal}
\def\symmetry{symmetry}
\def\Symmetry{Symmetry}

\def\minf{-\infty}

\def\ra{a}
\def\bfa{\textbf{\ra}}
\def\rb{b}
\def\bfb{\textbf{\rb}}

\def\bfi{ \textbf{i}}
\def\bfj{\textbf{j}}
\def\bfk{\textbf{k}}

\def\bft{\textbf{t}}
\def\bfs{\textbf{s}}

\def\bfo{\textbf{o}}

\def\bfv{\textbf{v}}

\def\Dl{\Delta}
\def\a{\al}

\newcommand{\pfi}[1]{\operatorname{N}({#1})}

\newcommand{\per}[1]{\operatorname{per}({#1})}

\def\eqM{\; \overset{\operatorname{w}}{=} \;}
\def\neqM{\; \overset{\operatorname{w}}{\ne}\;}

\def\tTrop{\tilde \Trop}

\def\eqR{\overset{_e}{\sim}}
\def\T{\Trop}
\def\tpF{f}

\def\semiDU{\overset{\operatorname{semi}}{\sqcup}}
\def\multU{\; \overset{\operatorname{w}}{\cup} \;}

\def\MultU{\; \overset{\operatorname{w}}{\bigcup} \; }
\def\mMinus{\setminus_{\operatorname{w}}}

\newcommand{\comp}[1]{{#1}^{\operatorname{spl}}}
\newcommand{\mincomp}[1]{{#1}^{\operatorname{min-spl}}}
\newcommand{\rvl}[1]{{#1}^{\operatorname{rvl}}}
\newcommand{\sym}[1]{{#1}^{\operatorname{sym}}}
\newcommand{\trn}[1]{{#1}^{\operatorname{trn}}}
\newcommand{\pSet}[2]{{#1}^{(#2)}}

\def\Var{W}
\def\Sur{H}
\def\Prm{P}
\def\Cur{C}
\def\Set{S}

\newcommand{\red}[1]{\widetilde{#1}}
\def\rVar{\red{W}}
\def\rSur{\red{H}}

\def\Lt{\Sigma}

\begin{document}


\title[Completions, Reversals, and Duality for Tropical Varieties] {Completions, Reversals, and Duality \\ for Tropical Varieties}


\author{Zur Izhakian}\thanks{The first author has been supported by the
Chateaubriand scientific post-doctorate fellowships, Ministry of
Science, French Government, 2007-2008.}
\thanks{The second author has been supported in part by the
Israel Science Foundation, grant 1178/06.}
\address{Department of Mathematics, Bar-Ilan University, Ramat-Gan 52900,
Israel} \address{ \vskip -6mm CNRS et Universit´e Denis Diderot
(Paris 7), 175, rue du Chevaleret 75013 Paris, France}
\email{zzur@math.biu.ac.il, zzur@post.tau.ac.il}
\author{Louis Rowen}
\address{Department of Mathematics, Bar-Ilan University, Ramat-Gan 52900,
Israel} \email{rowen@macs.biu.ac.il}

\date{October 2008}




\begin{abstract} We state and prove an identity for polynomials over the
max-plus algebra, which shows that any polynomial divides a
product of binomials. Interpreted in tropical geometry, any
tropical variety $\Var$ can be completed to a union of
hypersurfaces. In certain situations,  $\Var$ has a ``reversal''
variety, which together with $\Var$ already yields the union of
hypersurfaces; this phenomenon also is explained in terms of the
algebraic structure.
\end{abstract}

\maketitle



\section*{Introduction}\label{intro}

 Tropical mathematics  has been developed mainly
over the tropical semiring $\Trop_{\max} = \Real\cup\{-\infty\}$,
whose addition and multiplication are respectively  the operations
of maximum and summation,
$$a\TrS b=\max\{a,b\},\quad a\TrP b=a+b,$$
 cf.~\cite{Gathmann:0601322,IMS,MikhalkinBook,MikhalkinICM,pin98,RST}. Factorization in polynomials over $\Trop_{\max}$ is
notoriously difficult,
cf.~\cite{IzhakianRowen2007SuperTropical,Kim:142898}. One reason
is that different polynomials in $\Trop[ \la _1, \dots, \la _n]$
(viewed as  functions from $\Trop^{(n)}$ to $\Trop$) may act as
the same function over the max-plus algebra, when the values of
one monomial are dominated by  other monomials. Thus, we write $f
\eqR g$ to denote that polynomials $f$ and $g$ correspond to the
same function. Our main result in this paper is the following
 new identity of polynomials in  $\Trop[ \la _1, \dots, \la _n]$:

\begin{theorem}\label{permprime} Suppose $f = \sum_{i=1}^m f_i \in \Trop[\lm_1,\dots, \lm_n]$, for $m \ge 2$. Then
\begin{equation}\label{twofacts}
 \prod_{i < j } (f_i + f_j)  \eqR \big( \sum_i f_i \big) \big(\sum_{i < j }
f_i f_j\big) \cdots \big(\sum_{i }\prod_{j \neq i} f_j\big).
\end{equation}
(The right side is written as a product of $m-1$ terms.)
\end{theorem}

This result could be viewed as an extreme failure of unique
factorization, since \emph{every} polynomial~$f$ which is a sum of
at least three distinct monomials is part of a factorization that
is not unique. Specifically, if $f_i$ are the monomials of $f$,
then $f$ is a factor of  the product $\prod_{i \ne j} (f_i+f_j),$
as was seen in \cite[Theorem
12.4]{IzhakianRowen2007SuperTropical}. On the other hand,
Theorem~\ref{permprime} has a positive geometric interpretation --
Every tropical variety~$\Var$  can be ``completed'' to a variety
$\tP(\Var)$ comprised of various $k$-dimensional planes,  which in
turn can be decomposed into a union of varieties that can be
interpreted via \eqref{twofacts}.

Equation \eqref{twofacts} also gives rise to ``reversals'' of
tropical varieties and a duality in tropical geometry. The
motivation for the former came from Mikael Passare's talk,
Mathematisches Forschungsinstitut Oberwolfach, December 2007.

\section{The tropical polynomial algebra}

Since our basic result is algebraic, we need to consider the
underlying algebraic structure of the max-plus algebra $\T$. It is
easy to see that $\T$ is a semiring (which by definition satisfies
all of the axioms of an associative ring except for existence of
negatives), where $-\infty$ is the zero element of $\T$. In fact,
the existence of negatives fails spectacularly, since $a \TrS b =
\max\{a,b\}$ can never be $-\infty$ unless $a= b= -\infty$. We
refer to \cite{golan92} as a standard reference on semirings. Many
familiar notions in ring theory (subrings, ideals, polynomials,
etc.) carry over almost word for word to semirings, although the
lack of additive negatives makes the construction of a factor
semiring much less useful. In particular, given any semiring~$R$,
we can form the semiring of polynomials $R[\la]$, whose addition
and multiplication are induced by addition and multiplication of
the coefficients, where $\a_i, \bt_j \in R$:

$$\big(\sum _i \a_i \la ^i\big)+\big(\sum _i \bt_i \la ^i\big) = \sum _{i }
(\a_i +\bt_{i}) \la^i;$$
$$\big(\sum _i \a_i \la ^i\big)\big(\sum _j \bt_j \la ^j\big) = \sum _{i+j = k}
\a_i \bt_{k-j} \la^k.$$

Here we have reverted to the usual algebraic notation, although we
are always working over $\T.$ Thus, $f = \sum _i \a_i \la ^i$
denotes $\bigoplus_{ i \in \Net} \al_i \TrP \la ^i,$ and the
substitution $f(a)$ denotes $\bigoplus_{ i \in \Net} \al_i \TrP a
^i,$ where $a^i = a \TrP \cdots \TrP a,$ taken $i$ times.

Formally iterating the polynomial construction $n$ times enables
one to define the polynomial semiring $\T[ \Lambda] = \T[\la_1,
\dots, \la _n]$ in $n$ commuting indeterminates $\la_1, \dots,
\la_n.$ Elements of $\Trop[\lm_1,\dots,\lm_n]$ are called
\textbf{tropical polynomials}. A tropical  polynomial which is a
sum of precisely two different monomials is called a \textbf{
binomial}.

Summarizing,   any tropical polynomial can be written as
\begin{equation}\label{eq:polyToFunc1} \tpF = \bigoplus_{\bfi \in
\Omega} \al_\bfi \TrP \Lm ^\bfi \ \in \
\Trop[\lm_1,\dots,\lm_n]\backslash\{-\infty\},
\end{equation}
where $\Omega\subset\Int^n$ is a finite nonempty set of $n$-tuples
$\bfi =(i_1,\dots,i_n)$ with nonnegative coordinates, $\al_\bfi
\in\T$ for all $\bfi \in\Omega$, and $\Lm ^\bfi $ stands for
$\lm_1^{i_1}\TrP\cdots\TrP \lm_n^{i_n}$.

\subsection{The upper essential polynomial
semiring}\label{sec:Introduction}

Any tropical polynomial $\tpF \in \Trop[\lm_1,\dots,\lm_n]
\setminus \{ \minf \}$ determines a piecewise linear convex
function $\tfF:\Real^{(n)} \To \Real$, defined by:
\begin{equation}\label{eq:polyToFunc2} \tfF(\bfa) \ = \ \max_{\bfi \in\Omega}\{ \langle
\bfa ,\bfi \rangle+ \al_\bfi \} ,\quad  \bfa \in\Real^{(n)}\ ,
\end{equation}
where $\langle \, \cdot , \cdot \, \rangle$ stands for the
standard scalar product. The map $\tpF\mapsto\tfF$ is not 1:1.
This map can be viewed naturally as a homomorphism of semirings,
but we do not pursue that path here; instead, we look for a
canonical polynomial representing each of these functions.

\begin{definition}\label{def:essentialPart} A polynomial $g$ is \textbf{dominated by} a polynomial $f$ if
$g(\bfa) \le f(\bfa)$ for all $\bfa \in \Trop^{(n)}.$
  Suppose $f
= \bigoplus \al _\bfi \TrP \Lm ^\bfi,$ and $h = \al_\bfj \TrP \Lm
^\bfj$ is a monomial of $f$, and write $f_h = \bigoplus _{\bfi \ne
\bfj}\al _\bfi \TrP \Lm ^\bfi.$ We say that the monomial $h$  is
\textbf{(upper) inessential} if $h$ is dominated by $f_h$ (or, in
other words, $f(\bfa) = f_h(\bfa)$ for each $\bfa \in
\Trop^{(n)}$); otherwise $h$ is said to be \textbf{(upper)
essential}. Note that  $h(\bfa) \le f(\bfa)$ for each inessential
monomial $h$ and all $\bfa \in \Trop^{(n)}$.  The \textbf{(upper)
essential part} $\tilde f$ of a polynomial
 $f = \bigoplus \al _\bfi \TrP \Lm^\bfi$ is the sum of all essential monomials
of $f$, while its \textbf{inessential part} $\uf$ consists of the
sum of all inessential monomials of $f$. When $f = \tilde f$, we
call $f$ an \textbf{essential polynomial}.
\end{definition}
\noindent (Note that any  monomial by itself, considered as a
polynomial,
 is essential.)

Using this definition we say that two polynomials $f$ and $g$ are
\textbf{essentially equivalent}, written $f \eqR g$, if $\tilde f
= \tilde g$, that is if $ f(\bfa) = g(\bfa)$ for each $\bfa \in
\Trop^{(n)}$. Clearly, $\eqR$ is an equivalence relation which,
for convenience, we call  $e$-\textbf{equivalence}. We always
consider factorization up to $e$-equivalence; in other words, we
say $g$ \textbf{divides} $f$ if $gh \eqR f$ for some polynomial
$h$. (Otherwise one could make any polynomial irreducible by
adding some inessential monomial.)

\begin{definition} The \textbf{essential polynomial semiring}, $\tTrop[\lm_1,\dots,\lm_n]$, of
 $\Trop[\lm_1,\dots,\lm_n]$ is the set of essential polynomials,
where addition and multiplication are defined by taking the
essential part of the respective sum or product in
$\Trop[\lm_1,\dots,\lm_n]$. In other words, if $\oplus$ and
$\odot$ are the respective  operations in
$\Trop[\lm_1,\dots,\lm_n]$, we define
$$f+g = \widetilde {f\oplus g}, \qquad fg = \widetilde {f\odot
g}$$ to be the corresponding operations in the essential
polynomial semiring $\tTrop[\lm_1,\dots,\lm_n]$.
\end{definition}

We identify the essential polynomial semiring
$\tTrop[\lm_1,\dots,\lm_n]$ with the isomorphic semiring of
polynomial functions $\{ \tilde f: \Trop^{(n)} \to \Trop \}$.
Abusing notation slightly, we still write elements of
$\tTrop[\lm_1, \dots, \lm _n]$ as polynomials, although strictly
speaking, they are equivalence classes of polynomials.

Since the meaning should be clear from the context, we use the
same notation, $+$ and $\cdot$ \ , for the
 operations of the essential polynomial semiring and for $\Trop[\lm_1,\dots,\lm_n]$.

The next observation shows how inessential terms often may arise.

\begin{lemma}\label{lem:essential} Assume $f = f_1 +  f_2 + f_3\in \Trop[\lm_1,\dots,\lm_n]$, where \pSkip
$$ f_1 = \lm_{1}^{j_1+k}\lm_{2}^{j_2-k} \lm_{3}^{j_3}\cdots
\lm_{m}^{j_{n}}, \qquad f_2 = \lm_{1}^{j_1}\lm_{2}^{j_2}
\lm_{3}^{j_3}\cdots \lm_{n}^{j_{n}}, \qquad f_3 =
\lm_{1}^{j_1-k}\lm_{2}^{j_2+k} \lm_{3}^{j_3} \cdots
\lm_{n}^{j_{n}},$$ are monomials and  $k \leq \min \{j_1, j_2\}$
is a natural number. Then $f_2$ is inessential for $f$.
\end{lemma}
\begin{proof}
Pick $\bfa = (a_1,\dots, a_n) \in \Trop^{(n)}$ and assume $a_1
> a_2$; then $f(\bfa) = f_1(\bfa) > f_2(\bfa), f_3(\bfa)$.
Conversely, if $a_2 > a_1$, then $f(\bfa) = f_3 (\bfa) >
f_1(\bfa), f_2(\bfa)$. When $a_1 = a_2$, $f(\bfa) = f_1(\bfa) +
f_3(\bfa) = f_1(\bfa) + f_2(\bfa) + f_3(\bfa)$. In every case,
$f_2$  is inessential for $f$.
\end{proof}

For the next lemma, we let $\mathcal I$ denote the set of all
$m$-tuples $\bfi = (i_1, \dots, i_{m})$ for which each $0 \le i_u
<m$ and $\sum _{u=1}^{m} i_u = \binom m 2$. For any $ \bfi = (i_1,
\dots, i_{m}) \in \mathcal I$ and $0 \leq j \leq m-1$, we define
the $j$-index $\iota _j(\bfi)$ to be the number of $i_u$'s that
equal $j$; define $\iota (\bfi) = (\iota _{m-1}(\bfi), \dots,
\iota _{0}(\bfi))$.

 Let $S_m$
  denote the set of permutations of $(0,1, \dots, m-1).$   Thus, $\bfi \in S_m$ iff $\iota  (\bfi)
= (1,1,\dots, 1)$.

 We say $\bfi $ is \textbf{admissible} if for each number $k$, the
 sum of the largest $k$ components of $\bfi$ is at most $(m-1) + \dots +(m-k) = km - \frac {k(k+1)}2.$

\begin{remark} Lexicographically,  $\iota
(\bfi) \le (1,1,\dots, 1)$ for each admissible $\bfi \in \tI$.
(Indeed, the sum of the largest two components of $\bfi$ is at
most $2m-3,$ which means that at most one component is $m-1$, so
the first component of $\iota  (\bfi)$ is at most 1. We are done
unless it is 1, and conclude by induction on $m$.)
\end{remark}

 For each $\bfi\in \mathcal I,$ we define the
monomial $$h_\bfi = \la_1^{i_1}\cdots \la_{m}^{i_{m}} =
\Lm^\bfi.$$ For any permutation $\sg \in S_{m},$ we denote
$$h_\sg =  \la_1^{\sg (0)}\cdots \la_{m-1}^{\sg (m-1)}  =
\Lm^\sig.$$

\begin{lemma}\label{lem:essential2} The polynomial $p = \sum _{\sg \in S_m} h_\sg$ dominates $h_\bfi$ for each
admissible $\bfi \in \mathcal I.$
\end{lemma}
\begin{proof}  The proof is by reverse
induction on the lexicographic order of $\iota (\bfi).$ The
assertion holds by hypothesis when $\iota  (\bfi) = (1,1,\dots,
1)$. In general, if $\iota  (\bfi) <(1,1,\dots, 1),$ then some
$j$-index is 0, which implies that for some $j'<j,$ the $j'$-index
$\iota_{j'} (\bfi)\ge 2$; in other words, $\bfi$ has components $
i_s = i_t = j'$ for suitable $s \ne t$.

Take $\bfi'$ to be the $m$-tuple in which  $i_s = j'+1$ and   $i_t
= j'-1$ (with all other components the same as for $\bfi$), and
let $\bfi''$ be the $m$-tuple in which $i_s = j'-1$ and $i_t  =
j'+1$. By Lemma~\ref{lem:essential}, $h_\bfi $ is dominated by
$h_{\bfi'}+h_{\bfi''}.$

We claim that  $h_{\bfi'}$ and $h_{\bfi''}$ are admissible and
$\le (1,1,\dots, 1)$. Indeed, this is clear for $j'< j-1,$ since
then $\iota_j (\bfi')= \iota_j (\bfi'') = 0.$ Thus, we may assume
that $j' = j-1.$ By definition of admissibility, $\iota_{j-1}
(\bfi) \le 2,$ because if $\iota_{j-1} (\bfi) \ge 3,$ one would
have   the
 sum of the largest $k = m-j+2$ components of $\bfi$ would be $$(m-1) + \dots +(j+1)+0 +3(j-2),$$ which is
 greater than $ km - \frac {k(k+1)}2.$ On the other hand, by definition of $j'$, we have $\iota_{j-1}
(\bfi) \ge 2,$ so $\iota_{j-1} (\bfi) = 2.$ Since $\bfi'$
increases one of the exponent of $i_s$ from $j-1$ to $j$, we see
that $\iota_j (\bfi') =1$ and $\iota_{j-1} (\bfi')=0,$ proving
$\iota (\bfi')< (1,1, \dots, 1),$ as desired. Clearly $
\iota(\bfi'')=  \iota(\bfi'),$ since the roles of $s$ and $t$ are
interchanged, so $h_{\bfi''}$ also is admissible.

 Clearly, $\iota(\bfi') = \iota(\bfi'')$ is
of higher lexicographic order than $\iota(\bfi)$ (since
$\iota_{j'+1}(\bfi')= \iota_{j'+1}(\bfi)+1 $), so, by reverse
induction, $h_{\bfi'}$ and $h_{\bfi''}$ are both dominated by $p$,
implying $h_\bfi$ is dominated by $p$.
\end{proof}

\begin{remark}\label{lowess} Although we focus on the image of
$\tTrop[\lm_1,\dots,\lm_n]$
  under the natural map  $$\tTrop[\lm_1,\dots,\lm_n] \to \{\tilde f:
\Trop^{(n)} \To \Trop \},$$ this map loses the (upper) inessential
part of a polynomial. If we consider instead the min-plus algebra
$\dTrop=\Trop_{\min}$, in which $\max$ is replaced by $\min$, then
the essential monomials become the ones taking on minimal values;
let us call these \textbf{lower essential}. Then we get a map
$$\tdTrop[\lm_1,\dots,\lm_n]\To \{\tilde f: {\dTrop}^{(n)} \to
\dTrop \},$$ where now we define addition by taking the minimum
value instead of the maximum value, and thus
   lose the lower inessential part of each polynomial. We
preserve more information by considering both together, i.e.,
$\tTrop[\lm_1,\dots,\lm_n] \times \tdTrop[\lm_1,\dots,\lm_n]$,
viewed in $$\{\tilde f: \Trop^{(n)} \To \Trop \}\times \{\tilde f:
{\dTrop}^{(n)} \to \dTrop \}.$$ (Even so, one still loses
information such as $f_2$ in Lemma~\ref{lem:essential}.) Since
this viewpoint leads to more complicated notation, we put it aside
for the time being, but return to it later.\end{remark}

\subsection{The tropical Vandermonde matrix}

Our main tool in proving Theorem \ref{permprime} is the
\textbf{tropical Vandermonde  matrix} $V_f$ of an essential
polynomial $f = \sum_{i=1}^m f_i\in \tTrop[x_1,\dots,x_n]$, define
as the $m \times m$ matrix with entries $v_{ij} = f_i^{j-1}$.
Since the determinant is not available in tropical algebra
(because it involves negative signs), one uses the permanent
$$\per{{V_f}} = \sum_{\sig \in S_m}  f_1^{\sig(0)} \cdots  f_m^{\sig(m-1)},
$$
where $m$ denotes the number of (essential) monomials in $f$. We
can compute the permanent in two ways:

\begin{lemma}\label{lem:Van} If $V_f = (\lm_i^{j-1})$ is an $m \times m$ Vandermonde matrix
(for $f = \sum \lm_i)$, then
\begin{enumerate}
    \item $\per{V_f} \eqR
\prod_{i < j } (\lm_i + \lm_j) $ and,\pSkip
    \item $\per{V_f} \eqR (\sum_i \lm_i) (\sum_{i < j }
\lm_i \lm_j) \cdots (\sum_{i }\prod_{j \neq i} \lm_j).$
\end{enumerate}
\end{lemma}
\begin{proof}  Let $p = \per{V_f}$; then $p$ is a homogenous polynomial of degree $\frac{m(m-1)}{2}$
 in the $m$ indeterminates $\lm_1,\dots, \lm_m$. Moreover, $p$ is
 a sum of  the $m!$ monomials $h_\sg$,   each
corresponding to a single permutation $\sig \in S_m$; thus $p$ is
the polynomial of Lemma~ \ref{lem:essential2}, which says that $p$
dominates $h_\bfi$ for each admissible $\bfi \in \mathcal I.$

But it is easy to see that each monomial of   $q_1 =  \prod_{i < j
} (\lm_i + \lm_j)$ has the form $h_\bfi$ where $\bfi$ is
admissible (since the extreme case is for some  $\lm _i$ to have
degree
  $m-1$, in which case the next indeterminate has degree at
most $m-2$, and so forth), and thus $h_\bfi$  is dominated by $p$.
Since each monomial of $p$ appears in $q$, we get $p \eqR q_1.$

Likewise,  each monomial of $q_2 = \left(\sum_i \lm_i\right)
\left(\sum_{i < j } \lm_i \lm_j\right) \cdots \left(\sum_{i
}\prod_{j \neq i} \lm_j\right) $ clearly has the form $h_\bfi$
where $\bfi$ is admissible, and each monomial of $p$ appears in
$q_2$, implying $p \eqR q_2.$
\end{proof}

\subsection{Proof of Theorem \ref{permprime}}

The proof of Theorem \ref{permprime} now becomes quite
transparent:

\begin{proof} 
Specialize $\lm_i$ to $f_i$ and apply Lemma \ref{lem:Van}.
\end{proof}

 Algebraically, Theorem
\ref{permprime} shows that the  factorization of $\per{V_f} \in
\tTrop[\lm_1,\dots,\lm_n]$ into irreducible polynomials is not
unique.

\begin{example} Suppose $f= \lm_1^i + \lm_2^j + \al $, where $\al \in \Real$,  is a polynomial in $\Trop[\lm_1,\lm_2]$. Then $$V_f = \vvMat{0}{\al}{\al^2}{0}{\lm_1^i }{\lm_1^{2i}
}{0}{\lm_2^j }{\lm_2^{2j} } \quad \text{and}$$
$$\per{V_f} \eqR  (\lm_1^i + \lm_2^j + \al)(\al\lm_1^i + \al \lm_2^j + \lm_1^i  \lm_2^j) \eqR
(\lm_1^i + \lm_2^j )(\lm_1^i + \al)(\lm_2^j + \al) \  .
$$
This yields two different tropical factorizations  of $\per{V_f}$
into irreducible polynomials. (The right factorization  is a
binomial factorization.)
\end{example}

In tropical algebra, perhaps ``unique factorization'' is the wrong
emphasis, but rather we should emphasize factorization of
 $\per{V_f}$ into binomials. We pursue this avenue in the next
 section.

 We define the sums
\begin{equation}\label{eq:fnotation}
\begin{array}{lcl}
  \pSet{f}{1} & =  & \sum_{i } f_i,  \\
  \pSet{f}{2} & =  & \sum_{i < j } f_i f_j,  \\
  \pSet{f}{3} & = & \sum_{i < j < k} f_i f_j f_k, \\
  \vdots & & \vdots \\
   \pSet{f}{m-1}&  = & \sum_{i }\prod_{j \neq i} f_j, \\
                   \end{array}
\end{equation}
and write $\trn{f}$ for $ \pSet{f}{m-1}$ which we call the
\textbf{transpose polynomial} of $f$.

Also, we write $\bar f$ for $\prod_{i < j } (f_i + f_j) $. Under
this notation, one can rephrase Theorem \ref{permprime} in the
language of the essential tropical semiring.
\begin{corollary}\label{cor:permprime2} Suppose $f = \sum_{i=1}^m f_i$ is a polynomial in
$\tTrop[\lm_1,\dots,\lm_n]$ with monomials $f_i$. Then
\begin{equation}\label{eq:main}
\per{V_f}  =  \prod_{i=1}^{m-1} \pSet{f}{i} = \prod_{i < j } (f_i
+ f_j) = \bar f .
\end{equation}
\end{corollary}

\begin{remark}\label{grow}
\begin{enumerate} \eroman
    \item
 Since the factorization \eqref{eq:main} is the mainstay of this
paper, let us pause for a moment to consider the degrees of the
factors, in the situation where $f$ is completely homogeneous of
degree $d$.  Note that $\deg (\bar f) = d \binom m 2$, but $\bar
f$ can be factored into a product of $\binom m 2$ homogeneous
binomials, each of degree $d$.  On the other hand, each
$\pSet{f}{k}$ is completely homogeneous, of degree $dk\binom m k,$
and need not be reducible. \pSkip

\item  $\overline{\( \bar f \)}$ need not be $\bar f$. For example, taking
$$f = (\la_1 + \la _2)(\la _1 + \la_3) = \la _1^2 + \la_1 \la _2 +
\la _1 \la _3 + \la _2 \la _3,$$ we have $$\begin{array}{ll}
                                                       \bar f & = (\la_1^2
+\la_1\la_2)(\la_1^2 +\la_1\la_3)(\la_1^2 +\la_2\la_3)(\la_1 \la_2
+\la_1\la_3)(\la_1 \la_2 +\la_2\la_3)(\la_1 \la_3 +\la_2\la_3) \\&
= \la_1^3 \la_2 \la _3 (\la_1+\la_2)^2 (\la_1+\la_3)^2
(\la_2+\la_3)(\la_1^2+\la_2\la_3).
\end{array}$$
Computing $\overline{\( \bar f \)}$, one easily sees that it has
higher degree than $f$.

\pSkip

 \item Another example where $\overline{\( \bar f \) }\ne \bar f$, even when
 $f$ is symmetric in the $\la _i$:
 For
$f = \la_1 + \la _2 + \la _3$,
$$ \bar f = (\la_1 + \la _2)(\la _1 + \la_3)(\la _2 + \la_3), $$
and a similar computation to (ii) yields $\overline{ \( \bar f \)
}$ to be a product of powers of the $\la _i$ together with
binomials of the form $\la _i + \la_j$ and $\la_i^2 + \la _j \la
_k$.

\end{enumerate}

\end{remark}

\section{Geometric interpretation}

The main definition of a tropical variety is given in
\cite{MIK07Intro}, as Section 1.2E. It is convenient for us to
work with the equivalent definition of tropical varieties in terms
of sets of roots of tropical polynomials, in the sense of
\cite[Section 2]{RST}; see also
\cite{IMS,MikhalkinBook,M4,RST,Shustin2005}.

Given a tropical polynomial $f \neq \tUniS$  in
 $\tTrop[\lm_1,\dots,\lm_n]$, we denote by $Z_{\Trop}(f)$ the set of
points $\bfa \in\T^{(n)}$, on which  the value $f(\bfa)$ either
equals $-\infty$, or is attained by at least two of the monomials
of $f$; the set $Z_\T(f)$ is called \textbf{an affine tropical
hypersurface}, whose elements are called \textbf{zeros} or
\textbf{roots} of $f$. (Considering $f$ as a tropical function,
$Z_\T(f) \cap \Real^{(n)}$ can be viewed as the domain of
non-differentiability of $f$.) Then
\begin{equation}\label{eq:multToUni}
Z_\T(fg) = Z_\T(f) \cup Z_\T(g).
\end{equation}
For a finitely generated ideal $A=\langle
\tpF_1,\dots,\tpF_s\rangle\subset\T[\lm_1,\dots,\lm_n]$, the set
$$Z_\T(A) \ = \ \bigcap_{\tpF\in
A}Z_\T(\tpF)\subset\Trop^{(n)}$$ is called \textbf{an affine
tropical (algebraic) set}. Clearly,
$Z_\T(A)=Z_\T(\tpF_1)\cap\dots\cap Z_\T(\tpF_s)$.

This definition is consistent with the definition given in
\cite{RST}, in view of \cite{MIK07Intro}, and is a natural
framework for developing the connections between algebra and
geometry of polyhedral complexes.

When $\tpF$ contains at least two monomials, the nonempty set
$Z(\tpF) = Z_\Trop(\tpF) \cap \Real^{(n)}$ is called the
\textbf{tropical hypersurface} in $\Real^{(n)}$ defined by $\tpF$.
Similarly, for a finitely set $A=\{
\tpF_1,\dots,\tpF_s\}\subset\T[x_1,\dots,x_n]$, the set $Z(A)
=Z_\T(A)\cap\Real^{(n)}$ is defined to be a \textbf{tropical
(algebraic) set in $\Real^{(n)}$}.

\begin{corollary}\label{geom} By the left side of Equation
\Ref{eq:main}, the Newton polytope (see Section \ref{nsec2} below)
of $\bar f$ is a union of hyperplanes. In $\Real^{(2)}$ this is
called a zonotope, i.e. a Minkowski sum of a set of line segments.
\end{corollary}

In other words, Theorem \ref{permprime} says any tropical
hypersurface is contained in a union of hyperplanes, which in turn
can be decomposed into a union of hypersurfaces of the polynomials
on the right side of Equation~\Ref{eq:main}. This leads us to view
$\bar f$ as some sort of closure of $f$. But, in view of
Remark~\ref{grow}(iii), this process could continue further, so we
would rather take the polynomial which is the minimal product of
binomials that is $e$-divisible by $f$, and we call it the
\textbf{reduced closure} of $f$. To achieve this, we factor out
all binomials first. In other words, writing $f = gh$ where $g$ is
the product of the binomial factors of~$f$, the reduced closure is
$g \bar h.$

\section{Applications in tropical geometry}
 A \textbf{tropical set} is a finite
union of convex closed rational (i.e. defined over $\Rati$)
polyhedra.  The \textbf{dimension} of a tropical set is the
maximum of the dimensions of these polyhedra.    (One can also
show that all finite unions of convex closed rational polyhedra of
positive codimension are tropical sets.)

A  face of a polyhedral complex is \textbf{top-dimensional} if it
has maximal dimension (with respect to all other faces).
 A finite polyhedral complex is said to be of \textbf{pure
dimension} $k$ if each of its faces of dimension $< k$ is
contained in a top-dimensional face. Conversely, we say that a
face is \textbf{bottom-dimensional} if it has minimal dimension
(with respect to all other faces).

A tropical hypersurface $\Sur$ in $\Real^{(n)}$ is then a finite
rational polyhedral complex of pure dimension $(n-1)$ where its
the top-dimensional faces $\delta$ are equipped with positive
integral \textbf{weights} $m(\delta)$ so that, for each
$(n-2)$-dimensional face $\sigma$ of $\Sur$   the following
condition is satisfied, which is called the \textbf{balancing
condition} (written using the standard operations):
\begin{equation}\label{eq:balanceCond} \sum_{\sigma\subset\delta}m(\delta)n_\sigma(\delta)=0\
,\end{equation}
where $\delta$ runs over all $(n-1)$-dimensional faces of $\Sur$
containing $\sigma$, and $n_\sigma(\delta)$ is the primitive
integral normal vector to $\sigma$ lying in the cone centered at
$\sigma$ and directed by $\delta$ \cite{Einsiedler8311,RST}. The
weight, $m(\dl)$, of a face $\dl$ is also called the
\textbf{multiplicity} of $\dl$.

In general, we define a \textbf{$k$-dimensional tropical variety
in $\Real^{(n)}$} as a finite rational polyhedral complex of pure
dimension $k$, whose top-dimensional faces are equipped with
positive integral weights and satisfy condition
(\ref{eq:balanceCond}) for each face of codimension $1$. (This
definition is given in the sense of \cite{MIK07Intro}, which
includes that of \cite{RST}.)

\begin{definition}\label{def:semidisjoint}
Let $\tS = \{S_i \subset \Trop^{(n)} : i \in I \subset \Net \}$,
be a finite collection of tropical sets, and let $S_J = \bigcap_{j
\in J} S_j$, $J \subseteq I$. Denoting by $\delta_j$ the face of
maximal dimension in $S_j$, $j\in J$, containing $S_J$, we say
that $\tS$ is \textbf{semidisjoint} if for any $J \subseteq I$,
$\dim S_J < \delta_j$ for each $j \in J$. We denote the
semidisjoint union by $\semiDU$ \; .
\end{definition}
\noindent Clearly, a disjoint collection of tropical sets  is
semidisjoint.

In order to distinguish between the standard notation of union and
equality of sets to that which include weights we define:
\begin{definition} Two tropical varieties  $\Var \subset \Real^{(n)}$ and $\Var' \subset  \Real^{(n)}$ are
said to be  \textbf{weighted equal}, denoted $\Var \eqM \Var'$, if
they are identical as sets and each of their corresponding top
dimensional faces has the same weight. The \textbf{weighted union}
of tropical varieties $U \subset \Real^{(n)}$ and $U' \subset
\Real^{(n)}$, denoted $U \multU U'$, is defined to be $U \cup U'$
where the weight of a top-dimensional face $\dl$ is the sum of the
weights of the faces in $U$ and $U'$ that comprise $\dl$.
\end{definition}
 This definition of the weighted union satisfies
additivity under union, as well as   the balancing condition
\Ref{eq:balanceCond}.  With the definition we also have the
relation $$Z(fg) = Z(f) \multU Z(g),$$ for any  $f,g \in
\tTrop[\lm_1, \dots, \lm_n]$.

Analogously, we define the semidisjoint union with multiplicity.

\begin{example}
If $f \in \tTrop[\lm_1, \dots, \lm_n]$, then $Z(f) = Z(f^m)$ but
$Z(f) \neqM Z(f^m)$.
\end{example}

For the rest of this paper we only consider tropical varieties
that are also tropical algebraic sets; namely they can be written
as  $\Var = \bigcap \Sur_i$, i.e. a complete intersection, where
the $\Sur_i$ are tropical hypersurfaces. Moreover, all tropical
hypersurfaces are considered as tropical varieties, i.e. equipped
with weights.

\subsection{Tropical primitives}

\begin{definition}\label{def:primitive}
 A $k$-dimensional tropical variety of one face is called a \textbf{$k$-dimensional tropical primitive},
 or tropical primitive, for short, when the variety is a hypersurface.
\end{definition} \noindent Namely, a $k$-dimensional tropical primitive
is a degenerate tropical variety, which  in the classical sense is
simply a $k$-dimensional plane having rational slopes. One can
easily see that a $k$-dimensional tropical primitive is an
intersection of tropical primitives. By definition,   any
collection of different primitive hypersurfaces must be
semidisjoint.

\begin{remark}\label{rmk:prmTobino} Any $k$-dimensional tropical
 primitive $\Prm \subset \Real^{(n)}$ is a tropical
variety $Z(A)$, for which $A = \langle p_1, \dots, p_1 \rangle$ is
an ideal generated by tropical binomials. Writing $\Prm = \bigcap
\Prm_i$, with each $\Prm_i$   an $(n-1)$-dimensional tropical
primitive with rational slopes, say $\frac{t_1}{s_1}, \dots,
\frac{t_n}{s_n}$ with each $t_i, s_i \in \Net$, we can define the
binomial
$$p_i = \al_\bft \lm_1^{t_1} \cdots \lm_n^{t_n} + \al_\bfs \lm_1^{s_1} \cdots
\lm_n^{s_n}, \ \ \al_\bfs, \al_\bft \in \Real , $$ to get $Z(p_i)
= \Prm_i$.
\end{remark}

We say that a $k$-dimensional tropical variety is \textbf{generic}
if it does not have two or more top-dimensional faces contained in
some tropical primitive of dimension $k$. A tropical variety is
called \textbf{reducible} if it is a weighted union of tropical
varieties; otherwise is it called \textbf{irreducible}. In
particular, when $\Sur = Z(f)$ is a reducible hypersurface, then
$\Sur = Z(g) \multU Z(h)$ and $f = gh$ for some polynomials $g$
and $h$ in $\tTrop[\lm_1,\dots,\lm_n] $ (cf. \Ref{eq:multToUni}).
When $\Var = \Var' \multU \Prm$ is a tropical variety and $\Prm$
is some $k$-dimensional tropical primitive, we say $\Sur$ is
\textbf{primitively reducible}, otherwise $\Var$ is called
\textbf{primitively irreducible}.

\begin{lemma}\label{lem:irrInPrm} Any non-primitive tropical
hypersurface $\Sur \subset \Real^{(n)}$ which contains a tropical
primitive $\Prm$ is primitively reducible.
\end{lemma}

\begin{proof} Assume $\Prm$ is of
weight $m$.  Since $\Sur$  contains $\Prm$, then all of its top
dimensional faces which lying over $\Prm$ have weight $\geq m$
(not necessarily all of the same weight).  For any
$(n-2)$-dimensional face $\sig$ of $\Sur$ contained in $\Prm$,
  there are $(n-1)$-dimensional faces $\dl , \dl' \subset \Prm
\cap \Sur$ whose intersection is $\sig$, i.e.   $\dl \cap \dl' =
\sig$, and for which the balancing condition \Ref{eq:balanceCond}
is satisfied. In particular, $n_\sig(\dl) = - n_\sig(\dl')$ and
$m(\dl), m(\dl') \geq m$.   Reducing $m(\dl)$ and $m(\dl)$ by $m$,
Condition \Ref{eq:balanceCond} is still satisfied for all $\sig
\subset \Prm$, so we can erase $\Prm$ from $\Sur$ and denote the
result as $\Sur \mMinus \Prm$, which remains a tropical
hypersurface. (Note that some faces of $\Sur$ which lie on $\Prm$
might exist also in $\Sur \mMinus \Prm$, but with lower weights.)
\end{proof}
We call the procedure described in the proof \textbf{extracting a
primitive from} $\Sur$ and denote it $\Sur \mMinus \Prm$. (When
all the top-dimensional faces of $\Sur$ on $\Prm$  are of weight
$m$, equal to the multiplicity of $P$,  then $\Sur \mMinus \Prm =
\Sur \setminus \Prm$.) In view of Remark \ref{rmk:prmTobino},
assuming $\Sur = Z(f)$, extracting a primitive from $\Sur$ is
equivalent to cancelling a
 binomial factor from $f$.

Given a tropical hypersurface $\Sur$, we define the procedure of
\textbf{primitive reduction} by discarding  sequentially all the
possible primitives from $\Sur$, and call the result, $\rSur$, the
\textbf{reduced tropical hypersurface} of $\Sur$. (By
construction, the    primitive reduction procedure is independent
of the order of extraction, and thus is canonically defined.)
Accordingly, we say that two hypersurfaces $\Sur$ and $\Sur'$ are
 \textbf{equal modulo primitives} if their reductions are identical.

\begin{remark} From
a more general algebraic point of view, we could define the
semiring
$$ \tTrop\langle\langle\lm_1,   \dots, \lm_n\rangle\rangle,$$
whose elements are formal sums
$$\sum \al_\bfi \lm_1^{i_1} \cdots
\lm_n^{i_n}, \qquad \al_\bfi \in \Real, \ \  i_1, \dots, i_n \in
\mathbb Q,$$
where addition and multiplication are just as with polynomials.
When we substitute $a \in \Real$ tropically for~ $\lm$ in the
monomial $\lm ^{m/n},$ using the standard notation, we just take
$\frac{m}{n}a.$

A binomial $p = \al_\bfs \lm_1^{s_1} \cdots \lm_n^{s_n}+\al_\bft
\lm_1^{t_1} \cdots \lm_n^{t_n}  ,$ where now the $s_i,t_i \in
\mathbb Q,$  can be rewritten (tropically) as $$\(\frac
{\al_\bft}{\al_\bfs} \lm _1^{t_1-s_1}\cdots
\lm_n^{t_n-s_n}+0\)\al_\bfs \lm_1^{s_1} \cdots \lm_n^{s_n}.$$

Cancelling out the monomial on the right, we obtain a binomial of
the form $\al_\bfs \lm_1^{s_1} \cdots \lm_n^{s_n}+0$, which we say
has  \textbf{normal form}.

Given any binomial  $p = \al_\bfs \lm_1^{s_1} \cdots
\lm_n^{s_n}+0$ of normal form, we delete those $\lm _i$ for which
$s_i = 0,$ and thus  rewrite $p$ as $\al_\bfs \lm_1^{s_1} \cdots
\lm_m^{s_m}+0$, where $s_m \ne 0.$ The algebraic analog of
extracting a primitive is to take the semiring obtained by
identifying the two monomials of any binomial $p$. In order to do
this, we replace $\lm_m$ by $(\al_\bfs)^{-1/s_m}
\lm_1^{-s_1/s_m}\cdots \lm_{m-1}^{-s_{m-1}/s_m}.$ Performing this
elimination process the same way that one reduces indeterminates
in linear algebra via Gauss-Jordan elimination, effectively
reduces the number of indeterminates.

We believe that this is the ``correct'' way to view  tropical
geometry in terms of polynomials. \end{remark}

\begin{example} Let $f = \lm_1^2 + \lm_1 \lm_2 + \lm_1 + \lm_2 +
0 = (\lm_1 + \lm_2 + 0) (\lm_1 +0)$. Extracting a primitive from
$\Sur = Z(f)$, corresponds to cancelling the binomial factor $p =
\lm_1 +0$ from $f$ to get $\rSur = Z(\lm_1 + \lm_2 + 0)$.
\end{example}

Similarly, the reduction $\rVar$ of a $k$-dimensional tropical
variety $\Var$ is obtained by discarding all possible primitive of
dimension $k$, and
\begin{equation}\label{eq:modPrim}
    \Var  \equiv \Var' \text{ modulo primitives }  \IFF \rVar =
    \rVar'.
\end{equation}
Namely, each reduced tropical variety stands for a class of
varieties. (Note that $\rVar = \bigcap \rSur_i$ is not the
reduction of the tropical variety $\Var = \bigcap \Sur_i$.)
\begin{remark}\label{rmk:irrVar}
In view of Lemma \ref{lem:irrInPrm},  $\red{\Var} = \Var$ for any
irreducible, or primitively irreducible, tropical variety $\Var =
\bigcap \Sur_i$ .
\end{remark}

\begin{definition}\label{def:primitiveCovering}
A \textbf{primitive covering} of a tropical set $\Set \subset
\Real^{(n)}$ is a finite collection of $k$-dimensional  tropical
primitives $\tP(\Set) = \{ \Prm_i : \Prm_i \subset \Real^{(n)} \}$
whose union contains $S$. Denoting the cardinality of $\tP(\Set)$,
counting multiplicities,  by $|\tP(\Set)|$, we say that
$\tP(\Set)$ is a \textbf{minimal covering} of $\Set $ if
$|\tP(\Set)|$ is minimal over all the possible covers of $S$.
\end{definition}
\noindent (A primitive cover may contain overlapping primitives;
in this case the primitives are counted with their
multiplicities.)

Clearly, any tropical set $\Set \subset \Real^{(n)}$ has a
primitive cover, where $|\tP(\Set)| \leq$ the sum of all the
multiplicities of faces of  $\Set$. For a $k$-dimensional tropical
variety $\Var$, this upper bound  can be reduced to
$$|\tP(\Var)| \leq  \sum_\dl  m(\dl),$$
(the operations here are the standard operations) where $\dl$ runs
over all the $k$-dimensional faces of~$\Var$. Yet, this naive
upper bound often can be reduced much further.
%

\subsection{Starred  varieties} Among tropical varieties  we identify
a special family with a nice behavior which is much easier to
analyze.

\begin{definition}\label{def:starred }  A tropical variety $\Var \subset \Real^{(n)}$ is called \textbf{starred } if
it has a single    bottom-dimensional face.
\end{definition}
\noindent Accordingly, a tropical variety that has a primitive
cover, all of whose elements intersect at a single face, is
starred. (This definition  also includes cases in which varieties,
or hypersurfaces, do not have a proper $0$-dimensional face; for
example  $\Sur = Z(f)$, with $f = \lm_1 + 0$ in $\tTrop[\lm_1,
\lm_2]$, is starred of bottom-dimension~1.)

\begin{example} The following are straightforward  examples of  starred  varieties  in $\Real^{(n)}$:
\begin{enumerate}
    \item A tropical hyperplane (thereby permitting one to use starred varieties to answer questions raised
    in Passare's talk cited above), \pSkip
    \item A tropical primitive, \pSkip
    \item A tropical curve having a single
    vertex, \pSkip
    \item Example \ref{exm:1} below. \pSkip
\end{enumerate}
\end{example}

Locally, any tropical algebraic set $\Set \subset \Real^{(n)}$ is
a starred  variety. When a local neighborhood contains points of
only
 one face of $\Set$, then, locally, $S$ is a tropical primitive.

\begin{lemma}\label{lem:sSurTosVar}
Any tropical $k$-dimensional starred  variety $\Var = \bigcap
\Sur_i$ is the intersection of tropical starred  hypersurfaces.
\end{lemma}

\begin{proof}
Let $\tau$ be the single bottom-dimension face of $\Var = \bigcap
\Sur_i$, where $\Sur_i = Z(f_i)$. Then, $\tau \subset \bigcap_j
\dl_{i,j}$, where $\dl_{i,j}$ are   top-dimensional faces of
$\Sur_i$, each determined  by a pair of monomials $f_{i,\bfj }$
and $f_{i,\bfk }$ of $f_i$. In case one of the $\Sur_i$ is not
starred , one can replace it by the
 hypersurface determined by the pairs of monomials corresponding
 to the top-dimensional faces $\dl_{i,j}$
 (and discarding all the other monomials of~$f_i$).
\end{proof}

\subsection{Tropical hypersurfaces and subdivisions of their Newton polytopes}\label{nsec2} The convex
hull~$\Delta$ of the set $\Omega$ in the formula
\Ref{eq:polyToFunc1} (or, equivalently, in formula
\Ref{eq:polyToFunc2}) for a tropical polynomial $\tpF$ is called
the \textbf{Newton polytope of~$\tpF$}. The Legendre dual to
$\tpF$ is a convex piece-wise linear function
$\nu_{\tpF}:\Delta\to\Real$, whose maximal linear domains form a
subdivision
\begin{equation}\label{eq:polySub} S({\tpF})\ :\
\Delta=\Delta_1\cup\dots\cup\Delta_N
\end{equation}
into convex lattice polytopes of dimension
$\dim\Delta_i=\dim\Delta$, $i=1,\dots,N$. The vertices of the
subdivision $S({\tpF})$ bijectively correspond to the essential
monomials of $\tpF$; in particular, the vertices of $\Delta$
always correspond to essential monomials of $\tpF$. A subdivision
$S(f)$ is called an \textbf{empty subdivision} if it has no
interior vertices, i.e. vertices which are not vertices of
$\Delta$.

There is the following combinatorial duality,  between the finite
polyhedral complexes which inverts the incidence relation:
$\Delta$, covered by the faces of the subdivision $S({\tpF})$, and
 $\Real^{(n)}$, covered by the faces of the hypersurface $Z(\tpF)$ and
by the closures of the components of  $\Real^{(n)}\backslash
Z(\tpF)$.


 Namely: \pSkip
\renewcommand{\labelenumi}{(\alph{enumi})}
\begin{enumerate}
\item The vertices of $S({\tpF})$ are in one-to-one correspondence
with the components of $\Real^{(n)}\backslash Z(\tpF)$, so that
the vertices of $S({\tpF})$ on $\partial\Delta$ correspond to
unbounded components, and the other vertices of $S({\tpF})$
correspond to bounded components; \pSkip

\item A $k$-dimensional face of $S({\tpF})$, $k\ge 1$, corresponds
to an $(n-k)$-dimensional face of $Z(\tpF)$, and they are
orthogonal to each other. \pSkip
  \end{enumerate}
  \renewcommand{\labelenumi}{(\arabic{enumi})}

A tropical hypersurface $Z(\tpF)$ considered as a tropical variety
(i.e., equipped with weights) determines the Newton polytope
$\Delta$ and its subdivision $S({\tpF})$ uniquely, up to
translation in $\Real^{(n)}$, and determines the essential part
$\tilde f$ (i.e., the sum of the essential monomials) of the
tropical polynomial $\tpF$ up to multiplication by a monomial;
therefore,
$$S(f) = S(\tilde f).$$
On the other hand, as a polynomial, $f$ determines the Newton
polytope uniquely. Without weights, $Z(\tpF)$ determines the
combinatorial type of $\Delta$ and of its subdivision, together
with the slopes of all the faces of $S({\tpF})$.

\begin{note}\label{nt:polytopeOfBinomial}
Accordingly:
\begin{enumerate}
    \item Given a polynomial $f\in \tTrop[\lm_1,\dots, \lm_n]$ whose tropical
    hypersurface
    $Z(f)$ is of bottom dimension $k$, then its Newton polytope
    $\Delta$ is of dimension $n-k$. \pSkip

    \item  The Newton polytope $\Delta$ of an essential binomial
$p \in \tTrop[\lm_1,\dots, \lm_n]$ is simply a line segment  in
$\Real^{(n)}$ with empty subdivision $S(p)$. \pSkip

 \item  $Z(f)$ is starred  iff the subdivision $S(f)$ of the Newton
polytope $\Delta$   of $f$ is empty. \pSkip

\item If $\Delta$ has empty subdivision $S(f)$ and there exists
$(n-1)$-plane cut $\pi$ of $\Dl$ where all the $1$-dimensional
faces intersecting transversally with $\pi$ are parallel, then $Z(
f)$ contain a primitive. \pSkip

\end{enumerate}

\end{note}

Abusing language slightly, for a tropical hypersurface $\Sur =
Z(f)$, we sometimes say that $\Delta$ is the Newton polytope of
$\Sur$, and their faces are said to be dual in the sense described
above.

One approach to define the weights $m(\dl)$ of the top-dimensional
faces $\dl$ of a tropical hypersurface is by taking the integral
lengths of their dual one-dimensional faces in the subdivision of
the corresponding Newton polytope. For $(n-1)$-dimensional faces
these integral lengths, which are equal to the number of lattice
points  on the dual faces minus 1, and satisfy the balancing
condition~\Ref{eq:balanceCond}.

\begin{remark}\label{rmk:muti-dual}
This setting, in which weights are defined using integral lengths,
is canonical. Namely, the weights of the top-dimensional faces of
a tropical hypersurface $\Sur$ are independent of its polynomial
description; that is, even if $\Sur = Z(f) = Z(g)$, where $f \neq
g$, yet each top-dimensional face $\dl \subset \Sur$ has the same
weight. (Note that $f$ and $g$ need not to be $e$-equivalent.)

For example, take a polynomial $f = gh$, where $h$ is a monomial.
Then, $f \neqR g$.  On the other hand, $Z(f) = Z(g)$,  which
implies that the Newton polytope $\Dl$ of $f$ is an integral
linear translation of $\Dl'$, the Newton polytope of $g$; thus,
both $\Dl$ and $\Dl'$ determine the same weights for the
top-dimensional faces of $Z(f)$.
\end{remark}

Having the same weight setting as in Remark \ref{rmk:muti-dual},
let $\Prm = Z(p)$ be a tropical primitive, where
  $p \in \tTrop[\lm_1, \dots,\lm_n]$ is an essential binomial.  Assume that $p$ is rewritten as
  $$p =
\al_\bfk \Lm^\bfk (\al_\bfi \Lm^\bfi + \al_\bfj \Lm^\bfj)^m$$
with maximal possible $m \in \Net$, then the weight of $\Prm$
equals  $m$.

\begin{example} Recall the Frobenius rule $f^m = \sum_i f_i^m$,  for any polynomial $f= \sum_i f_i$
with monomials $f_i$ in $\tTrop[\lm_1,\dots, \lm_n]$, cf.
\cite[Theorem 2.40]{zur05Nullstellensatz}. Let $p =
\lm_1^m\lm_2^{m+j} + \lm_2^j$, with $m,j \in \Net$. Then
$$ p = \lm_2^{j} ( \lm_1^m\lm_2^m + 0) = \lm_2^{j}( \lm_1\lm_2 + 0)^m,$$
and thus $p$ has multiplicity $m$.

\end{example}
Therefore, any primitive cover can formed as a union of tropical
primitives, each  of multiplicity ~$m(P)$. Accordingly, we can
write $$|\tP(\Var)| = \sum_\Prm m(\Prm),$$ where $\Prm = Z(p)$
runs over all the primitives of $\tP(\Var)$  and $m(\Prm)$ are
their multiplicities as defined above.

\subsection{\Supplement s of tropical varieties}

\begin{definition}
A \textbf{\supplement } of a $k$-dimensional tropical variety
$\Var$ is a tropical variety $\comp{\Var}$ of dimension  $k$,
whose weighted  union with $\Var$ produces a primitive cover of
$\Var$, denoted by $\comp{\tP(\Var)}$, i.e.
\begin{equation}\label{eq:comp}
 \Var \  \MultU \ \comp{\Var} \  \eqM \ \comp{\tP(\Var)},
\end{equation}
is called the \textbf{\completion } of $\Var$.
  The \supplement\ (resp. completion) is said to be a
\textbf{pure \supplement} (resp. completion) when the weighted
union is a semidisjoint weighted union. We say that a \supplement\
(resp. \completion) is minimal when $ | \comp{\tP(\Var)} | $ is
minimal possible.
\end{definition}

Since the union is a weighted union, a primitive cover by itself
cannot be the \supplement\ of a tropical variety unless it is a
union of primitives. Conversely, the minimal \supplement\ of a
tropical primitive is the empty set.

Note that the \supplement\ of a $k$-dimensional tropical variety
$\Var \subset \Real^{(n)}$ is not its set-theoretic complement  in
the primitive cover $\tP(\Var)$, since the two sets are not
disjoint. In fact, $\Var \cap \comp{\Var} \neq \emptyset$ is a
collection of  faces of dimension $\leq k$.

 As will be seen later, the \supplement\ of a tropical
hypersurface $\Sur$ need not be of the same type as that of
$\Sur$. For example, the \supplement\ of a tropical hyperplane is
not a hyperplane. Moreover, the \supplement\ of an irreducible
hypersurface may be reducible,  also they might have different
combinatorial types.

\begin{lemma}\label{lem:mininalComp}
The minimal \supplement\ of a tropical variety $\Var$ is unique.
\end{lemma}
\begin{proof} First assume that $\Var$ is a tropical hypersurface $\Sur$. Assume $\comp{\Sur}_1$  and $\comp{\Sur}_2$  are two
different minimal \supplement s of a tropical hypersurface $\Sur$,
and let $\tP_i(\Sur)$ denote the corresponding primitive coverings
$\Sur \multU \comp{\Sur}_i$, $i = 1,2$.  Then, without weights
$\tP_1(\Sur) = \tP_2(\Sur)$; otherwise, one of the primitive
coverings has a primitive which does not contain a face of $\Sur$.
So, $\tP_1(\Sur)$ and $\tP_2(\Sur)$ have a common primitive with
unequal weight; say $m_1 > m_2$ respectively.  But, then one can
extract a primitive from $\tP_1(\Sur)$, thereby contradicting its
minimality.

In general, the case of tropical variety $\Var= \bigcap \Sur_i$
apply the same argument to possible $k$-dimensional primitive
coverings of $\Var$.
\end{proof}

We write $\mincomp{\Var}$ for the minimal \supplement\ of a
tropical variety $\Var$.
\begin{lemma}\label{lem:compOfComp} If  $\mincomp{\Sur}$ is the
minimal  \supplement\ of  a primitively irreducible undersurface
$\Sur$, then:
\begin{enumerate}
    \item $\Sur$ is the minimal \supplement \ of $\mincomp{\Sur}$. \pSkip

    \item $\mincomp{(\mincomp{\Sur})} = \Sur$.
\end{enumerate}
\end{lemma}

\begin{proof}
Follows directly  form the uniqueness of the minimal \supplement .
\end{proof}

\begin{corollary} Assume $\Var = \bigcap \Sur_i$ is a tropical variety,
where $\Sur_i \subset \Real^{(n)}$ are tropical primitively
irreducible hypersurfaces. Then, $\mincomp{(\mincomp{\Var})} =
\Var$.
\end{corollary}

When $\comp{\Sur} = Z(g)$, for some $g \in \Trop[\lm_1,\dots,
\lm_n]$, is the \supplement\ of a tropical hypersurface $\Sur =
Z(f)$, we also say that $g$ is a \supplement\ of $f$ and denote it
$\comp{f}$. (Note that $\comp{f}$ need not to be unique.)


\begin{theorem}\label{thm:supplementSur}
Any tropical hypersurface $\Sur \subset \Real^{(n)}$ has a
\supplement , $\comp{H} \subset \Real^{(n)}$,  which is also a
tropical hypersurface; when $\Sur$ is generic, then  its
\supplement\ is pure.
\end{theorem}
\begin{proof} Write $\Sur = Z(f)$ for some $f = \sum f_i$
in $\tTrop[\lm_1,\dots,\lm_n]$ and apply Corollary
\ref{cor:permprime2}. Denoting $\bar f  = \prod_{j  < \j} (f_i +
f_j)$, it is clear that $Z( \bar f ) =  \tP(\Sur)$ is a primitive
cover of $\Sur$, explicitly,  $\Prm_{i,j} = Z( f_i + f_j)$,
$\tP(\Sur) = \bigcup \Prm_{i,j}$.   Let
\begin{equation}\label{eq:supplementSur} g = \left(\sum_{i <
j } f_i f_j\right) \cdots \left(\sum_{i }\prod_{j \neq i}
f_j\right),
\end{equation}
and take $G = Z(g)$, clearly a tropical hypersurface. Using
Equation \Ref{eq:main}, we have $f g = h$ and thus $Z(f) \multU
Z(g) \eqM Z(h)$. Namely $\comp{\Sur} = G$ is a \supplement\ of
$\Sur$.

Assume that $\Sur$ is generic. Thus, on each primitive
$\Prm_{i,j}$ of the primitive cover $\tP(\Sur)$,  $\Sur$  has at
most one top-dimension face $\dl$, i.e. $\Prm_k \setminus \dl
\subset \comp{\Sur}$. So, all the intersections of $\Sur$ and
$\comp{\Sur}$ are of dimension $< (n-1)$.
\end{proof}

\begin{example} The \supplement \  \Ref{eq:supplementSur} of a tropical
hypersurface $\Sur = Z(f) \subset \Real^{(n)}$ whose Newton
polytope ~$\Delta$ is a  simplex  (and thus has empty subdivision)
is a minimal pure \supplement .

Indeed, $\Delta$ has $n+1$ vertices, each corresponding to a
monomial of $f$, and exactly $\chos{n+1}{2}$ 1-dimensional faces,
dual to the top-dimensional faces of $\Sur$. On the other hand,
the primitive cover of $\Sur$ consists of $\chos{n+1}{2}$
primitives (not counting multiplicities) which are determined by
the pairs of different monomials of $f$, cf. Theorem
\ref{permprime} (1). Thus, $\comp{\Sur} = Z(\comp{f})$ is a pure
\supplement .

To see that $\tP(\Sur) = \Sur \multU \comp{\Sur}$ is the minimal
cover, just note, by construction, that each primitive has the
same multiplicity as the top-dimensional face it covers.
\end{example}

\begin{corollary}\label{cor:coverOfVar} Assume $\Var = \bigcap \Sur_i$ is a tropical variety,
where $\Sur_i \subset \Real^{(n)}$ are tropical hypersurfaces.
Then, $\comp{\Var} = \bigcap \comp{\Sur_i}$.
\end{corollary}
\begin{proof} Each top-dimensional face $\dl$ of $\Var$ is the
intersection of top-dimensional faces $\dl_i$ of $\Sur_i$
contained in some $k$-dimensional  primitive $\Prm = \bigcap_i
\Prm_i$ with $\dl_i \subset \Prm_i$. The \supplement\ of each
$\dl_i$ is also in $\Prm_i$ and thus their intersection is
contained in $\Prm$.
\end{proof}

\subsection{Examples} In this subsection we present a few examples
of typical planar \supplement s.

\begin{example}\label{exm:1}

\begin{figure}[!h]
\setlength{\unitlength}{0.7cm}
\begin{picture}(10,7)(0,0)

\thicklines \drawline(6,3)(8,1) \drawline(6,3)(6,0.5)
\drawline(6,3)(3,4.5)

\drawline[-20](6,3)(4,5) \drawline[-20](6,3)(6,5.4)
\drawline[-20](6,3)(9,1.5)

\put(2.1,4.6){\scriptsize{$\lm _1\lm
_2^2+0$}}\put(5.5,0.2){\scriptsize{$\lm _1+0$}}
\put(7.7,0.6){\scriptsize{$ \lm _1\lm _2 + 0$}}
\put(5.9,5.6){$p_2$} \put(9.1,1.2){$p_3$}\put(3.7,5.2){$p_1$}

\put(6.1,2.9){\scriptsize{$(0,0)$}}
\end{picture}
\caption{Illustration for Example \ref{exm:1}. \label{fig:1}}
\end{figure}

A tropical planar curve with a single node.  Let $\Cur = Z(f)$,
where $f = \lm _1^2\lm _2 + \lm _1 + 0$.  Take $ \comp{f} = \lm
_1^3\lm _2 + \lm _1^2\lm _2 + \lm_1$. Then $\comp{\Cur} =
Z(\comp{f})$ is a \supplement\ (and also a point \symmetry, as
explained below) of $\Cur$ along $(0,0)$. The primitive cover is
determined by the binomials
   $p_1 = \lm _1^2\lm _2 + 2,\  p_2 = \lm _1^2\lm _2 + 0, \ p_3 = \lm _1 + 0$,
   yielding  the
   equality
$$ f \comp{f} = p_1 p_2 p_3 = \lm _1 + \lm _1^2 + \lm _1^2\lm _2 +  \lm _1^3 \lm _2 + \lm _1^4 \lm _2 + \lm _1^4 \lm _2^2 + \lm _1^5 \lm _2^2.$$
See Fig. \ref{fig:1}, where the dashed lines correspond to
$\comp{\Cur}$ and the solid lines correspond to $\Cur$. This is a
pure
 \supplement\ which is the minimal \supplement .
\end{example}

\begin{example}\label{exm:2}(see Fig. \ref{fig:2}).
\begin{figure}[!h]
\setlength{\unitlength}{0.8cm}
\begin{picture}(10,7)(0,0)

\thicklines
\drawline(5,4)(2,4)\drawline(5,4)(6,3)\drawline(5,4)(7,6)
\drawline(6,3)(6,0.5)\drawline(6,3)(8,5)

\dottedline{0.15}(5,4)(3,2)\dottedline{0.15}(5,4)(6,4)\dottedline{0.15}(5,4)(3,6)
\dottedline{0.15}(6,3)(8,1)\dottedline{0.15}(6,3)(4,1)\dottedline{0.15}(6,3)(6,4)
\dottedline{0.15}(6,4)(7.5,5.5)

\drawline[-30](6,4)(6,6)
\drawline[-30](6,4)(8,4)\drawline[-30](6,4)(3.5,1.5)
%

\put(2,4.2){\scriptsize{$\lm
_2^2+0$}}\put(5.5,0.2){\scriptsize{$\lm _1^2+0$}}
\put(4.4,3){\scriptsize{$\al \lm _1\lm _2
+0$}}\put(6.5,6.2){\scriptsize{$\al \lm _1 +\lm _2$}}
\put(7.7,5.2){\scriptsize{$ \lm _1 + \al \lm _2$}}
\put(2.8,6.2){$p_1$}\put(5.9,6.1){$p_2$}
\put(8.3,4){$p_3$}\put(2.7,1.7){$p_5$}\put(3.7,0.7){$p_4$}
\put(3.1,1.2){$p_6$}

\put(4,4.2){\scriptsize{$(-\al,0)$}}\put(6.2,2.8){\scriptsize{$(0,-\al)$}}
\end{picture}
\caption{Illustration for Example \ref{exm:2}. \label{fig:2} }
\end{figure}
A tropical conic with two vertices.

Let $\Cur = Z(f)$, where $f= \lm _1^2 + \lm _2^2 + \al \lm _1\lm
_2 + 0,$
 with $\al > 0$. The
 \supplement\ of $f$ consists of two components, drawn in dashed and
 dotted  lines for $Z(\pSet{f}{2})$ and $Z(\pSet{f}{2})$, respectively. The
 primitive cover is determined by the following binomials (which are obtained by
 taking the sums of all pairs of monomials of $f$): $p_1= \al \lm _1\lm _2 +
0$, $p_2= \lm _1^2 + 0$, $p_3 = \lm _2^2 +0$, $p_4 = \lm _1 + \al
\lm _2$, $p_5 = \al \lm _1 + \lm _2$, and $p_6 = \lm_1 + \lm_2$.
The \supplement\ here is pure but not minimal.
\end{example}

\begin{example}\label{exm:3}
\begin{figure}[!h]
\setlength{\unitlength}{0.7cm}
\begin{picture}(10,8)(0,0)

\thicklines
\drawline(5,4)(3.5,7)\drawline(5,4)(5,2)\drawline(5,4)(7,2)
\drawline(5,2)(7,2)\drawline(5,2)(3.5,0.5)\drawline(7,2)(10,0.5)

\dottedline{0.15}(5,4)(5.8,2.7)\dottedline{0.15}(5,4)(5,7)\dottedline{0.15}(5,4)(3,6)
\dottedline{0.15}(5,2)(5.8,2.7)\dottedline{0.15}(5,2)(3,2)\dottedline{0.15}(5,2)(5,0.5)
\dottedline{0.15}(7,2)(5.8,2.7)\dottedline{0.15}(7,2)(8.5,0.5)\dottedline{0.15}(7,2)(9.5,2)

\drawline[-30](5.8,2.7)(8.8,5.7)\drawline[-30](5.8,2.7)(6.9,0.5)\drawline[-30](5.8,2.7)(2.8,4.2)

\put(5,4.2){\scriptsize{$(\al,0)$}}\put(7,2.1){\scriptsize{$(0,\al)$}}\put(4,2.1){\scriptsize{$(0,0)$}}

\put(3.2,7.2){$p_1$} \put(5,7.2){$p_2$} \put(8.9,5.7){$p_3$}
\put(9.8,2){$p_4$}\put(10.1,0.5){$p_5$}\put(8.5,0){$p_6$}
\end{picture}
\caption{Illustration for Example \ref{exm:3}. \label{fig:3} }
\end{figure}
Let $\Cur = Z(f)$, $f= \lm _1^2\lm _2^2 + \al \lm _1\lm _2 + \lm_1
+ \lm_2$,
 with $\al > 0$, be a (generic) tropical  curve of genus 1 (see Fig. \ref{fig:3}). The
 \supplement\ of $f$  consists of two components (drawn in dashed and
 dotted lines). The supplement again consists of six primitives; thus
 the \supplement\ is minimal (and also pure).
\end{example}

\subsection{The reduced \completion}
Once we have specified a \supplement\  $\comp{\Sur}$ of
hypersurfaces, with an explicit algebraic description, cf. Theorem
\ref{thm:supplementSur} we define the \textbf{reduced \supplement
} $\red{\comp{\Sur}}$ by taking the primitive reduction of
$\comp{\Sur}$. In this sense, for a tropical variety $\Var$, we
minimize $|\comp{\tP(\Var)}|$ as much as possible. The weighted
$\Sur \MultU \red{\comp{\Sur}}$ union is called \textbf{reduced
\completion } of $\Sur$.

\begin{remark} In the case of a hypersurface (the algebraic set of a
polynomial $f$), its reduced completion is the algebraic set of
the reduced closure of $f$.\end{remark}

\begin{claim}\label{clm:minimalRedCom}
$\red{\comp{\Sur}}$ is the minimal \supplement\ of the tropical
surface $\Sur$.
\end{claim}
\begin{proof} The primitive  reduction procedure discards only primitives,
whose set-theoretic \completion s are always the empty set, so
$\red{\comp{\Sur}}$ is a \completion\ as well. To see that it is
minimal, apply Lemma~\ref{lem:mininalComp}.
\end{proof}

\begin{corollary} The minimal \supplement \ of a tropical variety  $\Var = \bigcap_i
\Sur_i$ is $\red{\comp{\Var}} = \bigcap \red{\comp{\Sur}}$.
\end{corollary}
\begin{proof} Immediate from Corollary \ref{cor:coverOfVar}.\end{proof}

\begin{remark}\label{rmk:reducedComp} Suppose $\Sur$ is a primitively irreducible surface, i.e., $\Sur = \red{\Sur}$.
Denote the minimal primitive cover correspond to
$\red{\comp{\Sur}}$ by $\comp{\red{\tP}(\Sur)}$. Since
$\red{\comp{\Sur}}$ is the minimal \supplement \ of $\Sur$, the
multiplicity of each primitive $\Prm$ in $\comp{\red{\tP}(\Sur)}$
is equal to the sum of the weights of the top-dimensional
overlapping faces of $\Sur$ and $\red{\comp{\Sur}}$ that   $\Prm$
covers.
\end{remark}

Locally, any tropical variety $\Var \subset \Real^{(n)}$ can be
viewed as a starred  variety. Given a point $\bfa \in \Var$,
taking a small enough neighborhood $B(\bfa) \subset \Real^{(n)}$
of $\bfa$ and restricting  $\Var$ to $B(\bfa)$, one can see that
locally $\Var$ is either a starred  or a primitive variety. The
latter situation is trivial,  and we are mostly interested in the
starred varieties.

\begin{claim} The reduced \completion\ of a tropical starred  variety
$\Var \subset \Trop^{(n)}$ is a pure \completion .
\end{claim}

\begin{proof}
Clear from the fact that $\Sur$ is starred.
\end{proof}

Let $\tau \subset \Sur$ be a bottom-dimensional  face, $\bfa \in
\tau$ a point, and  $B(\bfa)$  a small neighborhood. Denoting  the
restriction of $\Sur$  to $B(\bfa)$ by $\Sur_\tau$; then
$\Sur_\tau$ is a starred  hypersurface of the same
bottom-dimension as $\Sur$. Let ${\comp{\Sur_\tau}}$ be its
\completion . Constructing the \supplement\ locally and viewing it
in $\Real^{(n)}$, we have the following identification:
\begin{theorem}\label{thm:LocalGlobal}
The global reduced \supplement \ of a generic tropical
hypersurface $\Sur$ is equal to the primitive reduction of the
weighted union of the local \supplement s along its
bottom-dimension faces, i.e.
$$ \red{\comp{\Sur}} \eqM \red{\overset{\operatorname{w}}{\bigcup_\tau} \; \comp{\Sur_\tau} }$$
where $\tau$ runs over all the  bottom-dimensional faces of
$\Sur$.
\end{theorem}
\begin{proof} Each of the faces of $\Sur$ contains at least one of
the bottom-dimensional faces $\tau$; thus, it enough to take the
\supplement\ along these faces to get a \completion\ of $\Sur$.
Taking the primitive reduction of
$\overset{\operatorname{w}}{\cup_\tau} \; \comp{\Sur_\tau}$ we get
a minimal \supplement\ of $\Sur$, which is unique by Lemma
\ref{lem:mininalComp} and thus equal to $\red{\comp{\Sur}}$; cf.~
Claim  \ref{clm:minimalRedCom}.
\end{proof}

Given a bottom-dimensional face $\tau$ of a tropical  variety
$\Var = \bigcap_i \Sur_i$, locally $\Var_\tau = \bigcap_{i,\sig}
\Sur_{i,\sig}$, where $\sig$ is a top-dimensional  face of
$\Sur_i$ that contains $\tau$. Combining Corollary
\ref{cor:coverOfVar} and Theorem \ref{thm:LocalGlobal} we
conclude:

\begin{corollary} Given a generic tropical variety  $\Var = \bigcap_i
\Sur_i$, then 
$$ \red{\comp{\Var}} \eqM
\red{\overset{\operatorname{w}}{\bigcup_\tau} \; \comp{\Var_\tau}
} \ , $$
where $\tau$ runs over all the    bottom-dimensional faces of
$\Var$.
\end{corollary}

\subsection{\Supplement  al  duality}  A dual correspondence for
tropical hypersurfaces is established by taking the reduced
\supplement .

\begin{theorem} The reduced \supplement\ of a tropical hypersurface admits
a duality; i.e., for any $\Sur \subset \Real^{(n)}$
\begin{equation}\label{eq:dualCom}
 \red{\comp{\OP\red{\comp{\Sur}}\CP}} \eqM \red{\Sur}.
\end{equation}
\end{theorem}
\begin{proof} We may assume that $\Sur$ is primitively irreducible.
$\red{\comp{\Sur}}$ is the minimal \supplement\ of $\Sur$; cf.~
Claim~\ref{clm:minimalRedCom}. Conversely, $\Sur$ is minimal
\supplement\ of $\red{\comp{\Sur}}$, cf. Lemma
\ref{lem:compOfComp}. On the other hand
$\red{\comp{\red{\comp{\Sur}}}}$ is also minimal \supplement\ of
$\red{\comp{\Sur}}$, which is unique, cf. Lemma
\ref{lem:mininalComp}.
\end{proof}

\begin{example} The dual curves for the Examples \ref{exm:1}
 and \ref{exm:3} are precisely their \supplement s, which we recall are
 minimal; see Figs \ref{fig:1} and \ref{fig:3} respectively.
The dual curve of Example \ref{exm:2} is obtained by extracting
$Z(p_6)$ from the  \supplement \ of $Z(f)$, drawn in dotted and
dashed lines, see Fig \ref{fig:2}.

\end{example}

\begin{corollary}\label{dor:dualVar} Given a tropical variety  $\Var = \bigcap_i
\Sur_i$ then  $ \red{\comp{\OP\red{\comp{\Var}}\CP}}\eqM
\red{\Var}$.
\end{corollary}

We call the relation in  Corollary \ref{dor:dualVar}, the
\textbf{\supplement  al\ duality} of tropical varieties. This
duality is quite general; note that although the dual of a variety
has the same dimension,  a variety need not to be of the same type
as its dual. For example:
\begin{itemize}
 \item The dual of an irreducible variety might be
irreducible, or vise versa, cf. Example \ref{exm:2}; \pSkip
    \item  The dual of a curve of genus 1 (which, in tropical sense, is not a
rational variety) can be a rational cure, cf. Example \ref{exm:3}.
\pSkip
\item  a tropical variety and its dual might be of different combinatorial types,
cf. Example \ref{exm:3}. \pSkip
\end{itemize}

\section{The reversal isomorphism and its geometric interpretation}

The \supplement \  can be understood better from the decomposition
of Formula \eqref{eq:main}, by means of another algebraic tool.
Before introducing this tool, we pass to a more convenient
semiring than the polynomial ring. The motivation is that the
monomial $\al_\bfi \Lm^\bfi$ has no tropical roots other than
$\minf$. Thus, when considering nonzero roots (or when studying
projective tropical geometry) one could multiply or divide the
polynomials defining the variety by powers of the $ \lm_i$ without
affecting the variety. This observation often enables us to
``clean up'' some of the computations, by erasing powers of the
$\la_i$ that arise for example in Remark \ref{grow}(ii),(iii).

Accordingly, it is just as natural to work with
 the semiring $\Trop[\lm_1, \lm_1^{-1}, \dots, \lm_n, \lm^{-1}_n]$ of \textbf {Laurent
 polynomials},
 defined
just as polynomials except that powers of the $\la_i$ may be taken
to be negative integers as well. Given any Laurent polynomial,
denoted  $f(\la_1, \dots, \la_n),$ one can define the natural
substitution $f(\bfa),$ for $\bfa =  (a_1,\dots,a_n) \in
\Trop^{(n)}.$ (Indeed, $a_k^{-i_k}$ can be viewed as the inverse
of $a_k^{i_k}$.)

\begin{remark}\label{algebrset}
 There is a natural isomorphism, which we denote as
\begin{equation}\label{eq:dual*}
\dual : \Trop[\lm_1, \lm_1^{-1}, \dots, \lm_n, \lm^{-1}_n]  \ \To
\ \dTrop[\lm_1, \lm_1^{-1}, \dots, \lm_n, \lm^{-1}_n],
\end{equation}  given by
$\al \mapsto \al^{-1}$ for each $\al \in \Trop$ and $\lm_i \mapsto
\lm_i^{-1}$ for each $i$. (Thus, for any monomial $h$, $h^* =
\frac 1h.$) We call $f^*$ the \textbf{reversal} of $f$. Clearly,
by definition, $(f^*)^* = f$, so we have a duality, which also
induces a duality of the geometry.

To understand the connection between the algebra and the geometry
here, we note that $\dual$ reverses the order of values in the
monomials, and thus switches (max, plus) with (min, plus).
\end{remark}

Given a finitely generated ideal $I = \langle f_1,
\dots,f_m\rangle$, we write $I^*$ for  $\langle f_1^*,
\dots,f_m^*\rangle$ and call it the \textbf{reversal ideal} of
$I$.

When $f = \sum f_i $ , the product of the $f_i$'s is denoted
\begin{equation}\label{eq:fnotation2}
 \pfi{f} = \prod_i f_i \ .
 \end{equation}

\begin{example}
If $f =  f_i+f_j$ for monomials $f_i, f_j$, we have
$(f_i+f_j)^\dual = f_i^* +  f_j^*  = \frac 1{f_i}+\frac 1{f_j} =
\frac{f_i + f_j}{f_i f_j},$ which has the same variety as $f$.
Thus, the reversal of a binomial has the same variety as the
binomial.

 More generally, writing $f = \sum_{i=1}^m f_i $ as a sum of monomials
$f_i$,  recalling the notation $$\trn{f} = f^{(m-1)} = \sum _i
\prod _{j \ne i} f_j, \qquad \overline{f} = \per{V_f} =
\prod_{j<i}(f_i + f_j),$$ we have
$$
\begin{array}{lllllll}
 \overline{\fdual{f}} &  = & \prod_{i \neq j } \OP f_i^* +  f_j^*
\CP  & =  &  \prod_{i \neq j } \frac{f_i + f_j}{f_i f_j} \ \ = \ \
\frac{\bar f}{ \pfi{f}^{m-1}} \ ; \\[1mm]%
f^{(m-u)} & = & \prod \frac{\pfi{f}}{f_{i_1}\cdots f_{i_u}} & = &
 \pfi{f}^{\binom m u} (f^*)^{(u)} \quad  \forall u \ ; \\[1mm]
\trn{f} & = & f^{(m-1)} & = &  \pfi{f} \fdual{f} \ ; \\[1mm]
\overline{\trn{f}}  & =  &  \prod_{i \neq j } \OP \sum
\frac{\pfi{f}}{f_i}  \CP  &  = &
\pfi{f}^{\frac{n(n-1)}{2}}\overline{\fdual{f}}. \end{array}$$

\end{example}

To understand the connection between the algebra and the geometry
here, we note that $\dual$ reverses the order of values in the
monomials; thus, $Z(\fdual{f}) = Z(f^{(m-1)}).$

\begin{definition} The \textbf{reversal}  of a tropical hypersurface $\Sur =
Z(f)$ is defined as $\rvl{\Sur} = Z(f^*)$,  with $f^*$ as defined
above.  The reversal $\rvl{\Var}$ of a tropical variety $\Var =
\bigcap \Sur_i$ is then $\rvl{\Var} = \bigcap \rvl{\Sur_i}$.
\end{definition}
Note that reversals are defined for any tropical variety $\Var$,
not necessarily for starred or irreducible varieties. In
particular, if $\Var$ is a tropical primitive, then $\Var =
\rvl{\Var}$ and $\Var$ is self dual. From Remark \ref{algebrset}
we can conclude immediately:
\begin{corollary}\label{cor:rvlravl}
If $\Sur$ is a tropical surface, then $\rvl{(\rvl{\Sur})}  =
  \Sur.$ When $\Var = \bigcap \Sur_i$ is a tropical variety,  $\rvl{(\rvl{\Var})}  =
 \Var.$
\end{corollary}
We call the relation in  Corollary \ref{cor:rvlravl}, the
\textbf{\reversal \ duality} of tropical varieties. Clearly, for
$\Var$ and $\rvl{\Var}$ we have the following properties:
\begin{itemize}
    \item $\Var$ and $\rvl{\Var}$ are isomorphic of the same dimension, \pSkip

    \item they are of a same combinatorial type,  \pSkip

    \item the weights of top-dimensional reversal faces of $\Var$ and $\rvl{\Var}$ are equal.

\end{itemize}

\begin{remark} In  Remark \ref{lowess}, we considered $\tTrop[\lm_1,\dots,\lm_n]
\times \tdTrop[\lm_1,\dots,\lm_n]$. In light of
Remark~\ref{algebrset}, to see the entire picture, one should view
this in $\tTrop[\lm_1,\lm_1^{-1},\dots,\lm_n,\lm_n^{-1}] \times
\tdTrop[\lm_1,\lm_1^{-1},\dots,\lm_n,\lm_n^{-1}]$, in which the
isomorphism $\dual$ becomes an automorphism of degree 2.
\end{remark}

For deformations of surfaces, a geometric approach to duality has
been suggested recently by Nisse~\cite{Nisse}.

\section{\Symmetry \ of varieties}


We call a $k$-dimensional plane (in the classical sense) in
$\Real^{(n)}$, for short, $k$-plane. Let us state the definition
of symmetry which is used in this paper:
\begin{definition} A set $\Set \subset \Real^{(n)}$ is said to be
\textbf{point symmetric} if there exists  a point $\bfo \in
\Real^{(n)}$ for which, in the standard notation,
\begin{equation}\label{eq:ptSym} \bfa \in \Set \Dir 2 \bfo - \bfa \in
\Set,
\end{equation}
for any $\bfa \in \Set$.

A set is \textbf{k-plane symmetric}, $ k < (n-1)$, if all of its
restrictions to $(n-k)$-planes orthogonal to a fixed $k$-plane
$\pi$ are point symmetric (with respect to a point $\bfo \in
\pi$). We say that a set $\sym{S}$ is a $k$-plane
\textbf{\symmetry } of $\Set$ if their union is $k$-plane
symmetric.
\end{definition}

Let us describe explicitly the action of the isomorphism
\eqref{eq:dual*} on a monomial $f_i = \al_\bfi \Lm^\bfi $:

\begin{equation}\label{eq:minusMon}
    f_i^{-1}(\bfa) = \frac 1 {\al_\bfi \bfa^\bfi} = \frac{f_i(\bfa^{-1})}{ \al_\bfi^2} \ .
\end{equation}
Tropically, we also write $\bfa^2$ for  $(a_1^2,\dots, a_n^2)$ and
have the relation
\begin{equation}\label{eq:invMon}
 f_i\OP\bfa^2 \CP = \al_\bfi \bfa^{2\bfi} =
\frac{f_i(\bfa)^2}{ \al_\bfi},
\end{equation}
for any $\bfa, \bfb \in \Real^{(n)}$ .

\begin{lemma}\label{lem:sym} Let $\Sur = Z(f) \subset \Real^{(n)}$, $f = \sum_i \al_\bfi \Lm_\bfi$,
be a  primitively irreducible starred  hypersurface of bottom
dimension $0$. Then $\trn{\Sur} = Z(\trn{f})$ is the point
\symmetry \ of $\Sur$, and vise versa.

\end{lemma}

\begin{proof} Let $\bfo \in Z(f)$, $f = \sum f_i$, be the face of
bottom dimension $0$, in particular
\begin{equation}\label{eq:oVal} f(\bfo)
= f_1(\bfo)= \cdots = f_m(\bfo) = c.
\end{equation}
Assume $\bfa \in Z(f)$, then $f(\bfa) = f_i(\bfa)= f_j(\bfa)$ for
some $i$ and $j$.   Since $Z(\trn{f}) = Z(\sum_i f_i^{-1})$, we
can rewrite condition \Ref{eq:ptSym} tropically, and need to prove
$\frac{\bfo^2}{\bfa} \in Z(\sum_i f_i^{-1}) $. Indeed, using
Equations \Ref{eq:minusMon} and \Ref{eq:invMon} we have,
$$\sum_i f_i^{-1}\left(\frac{\bfo^2}{\bfa}\right)
%
%
 =
\sum_i \al_\bfi^{-2} f_i \OP \frac{\bfa}{\bfo^{2}}\CP =
\sum_i \al_\bfi^{-1}   \frac{f_i(\bfa)}{f_i(\bfo^{2})}
= \sum_i  \frac{f_i(\bfa)}{f_i(\bfo)^2}.
$$
But, $f_i(\bfo)^2 = c^2$ for each $i$,  cf.  Equation
\Ref{eq:oVal}, and this completes the proof.
 \end{proof}

\begin{theorem}\label{thm:symmOfSur} The  primitively irreducible tropical hypersurface $\trn{\Sur}= Z(\trn{f})$ is
the $k$-plane \symmetry\ of a tropical starred  hypersurface $\Sur
= Z(f)$, i.e. $\trn{\Sur} = \sym{\Sur}$.
\end{theorem}
\begin{proof} $\Sur$ is starred, thus has a single face  $\tau \subset \Sur$ of bottom
dimension, which is a $k$-plane. Consider its orthogonal
$(n-k)$-planes, and apply Lemma \ref{lem:sym} to the restrictions
of $\Sur$ to these $(n-k)$-planes.
\end{proof}

\begin{note} In the case of $\Real^{(2)}$, for starred  curves, the
minimal \supplement\ and the point \symmetry\ coincide, and thus
by Theorem \ref{thm:symmOfSur} we provide the explicit polynomial
that determines the minimal pure \supplement\ for this class of
curves, i.e. $\comp{\Cur} = Z(\trn{f})$. But, in dimension $3$ or
higher, $\sym{\Sur}$ is purely contained in $\comp{\Sur}$ even for
the simplest case of a non-degenerate tropical hyperplane (i.e.
non-primitive hyperplane).
\end{note}

\begin{corollary}
If $\Var = \bigcap \Sur_i$ is a tropical starred   variety, where
$\Sur_i$ are tropical hypersurfaces,   then $\sym{\Var} = \bigcap
\sym{H_i}$.
\end{corollary}
\begin{proof} The bottom-dimensional  face of $\Var$ is contained
in the intersection of the bottom-dimensional faces of $\Sur_i$.
Apply Theorem \ref{thm:symmOfSur} to each $\Sur_i$, and consider
the intersection of their symmetries.
\end{proof}
\begin{corollary}\label{cor:symsym}
Suppose $\Var$ is a tropical starred variety, then
$\sym{(\sym{\Var})} \ =
 \ \Var.$
\end{corollary}
We call the identification in  Corollary \ref{cor:symsym} the
\textbf{\symmetry \ duality} of starred tropical varieties; for
this duality we have the following properties:
\begin{itemize}
    \item $\Var$ and $\sym{\Var}$ are isomorphic of the same dimension, \pSkip

    \item $\Var$ and $\sym{\Var}$  are of the same combinatorial type,  \pSkip

    \item the weights of top-dimensional symmetric faces of $\Var$ and $\sym{\Var}$ are equal.

\end{itemize}
(The \symmetry \ duality need coincide with the \supplement al
duality only in the case of starred varieties in $\Real^{(2)}$.)

By proving Theorem \ref{thm:symmOfSur} we have also proved the
following  identity of polynomials:
\begin{theorem}\label{thm:strPoly} Suppose $f = \sum_{i=1}^m f_i$ is a polynomial in
$\tTrop[\lm_1,\lm_2]$    with monomials $f_i$ whose Newton
polytope $\Delta$ has  empty subdivision. Then
\begin{equation}\label{eq:id2} f \trn{f}  =
   \(\sum f_i \) \(\sum_i \prod_{j\neq i }f_j \)
    = (f_1 + f_2)(f_2+f_3) \cdots (f_{m-1} + f_{m})(f_{m} + f_{1})
    \, ,
\end{equation}
where the $f_i$'s are ordered according the order of the
corresponding vertices on the Newton polytope of $f$.
\end{theorem}
Here, the Newton polytope $\Delta$ is a polygon whose vertices all
lie on the boundary $\partial\Delta$ of $\Delta$, corresponding to
monomials $f_i$ of $f$. These vertices, and respectively the
$f_i$'s, can be labelled according to their order on
$\partial\Delta$. So, in the theorem, any pair $f_i$ and $f_{i+1}$
correspond to adjacent vertices on $\partial\Delta$.

\section{Symmetry of lattice polytopes}

A lattice polytope $\Delta$ is a polytope whose vertices are
lattice points of a lattice $\Lt$. Assume $\Lt = \Int^{(n)}$ is a
lattice embedded in $\Real^{(n)}$, and $\Dl$ is a lattice polytope
on $\Lt$; often an integral  linear translation, we may assume
that $\Dl$ is a lattice polytope on $\Net^{(n)} \hookrightarrow
\Real_+^{(n)}$. So we may assume that $\Lt = \Net^{(n)}$.

Given a lattice polytope $\Dl \in \Net^{(n)}$, one can assign to
$\Dl$ a tropical polynomial $f = \sum_i f_i $ in
$\tTrop[\lm_1,\dots,\lm_n]$, whose Newton polytope is $\Dl$.
Indeed, to any vertex $\bfv = (v_1, \dots, v_n)$ of $\Dl$ assign
the monomial $f_i = \lm_1^{v_1} \cdots \lm_n^{v_n} $. Thus, a
lattice polytope can be regraded as a Newton polytope (with empty
subdivision).

Assume $Z(f) = \Sur \subset \Real^{(n)}$ is a primitively
irreducible tropical starred hypersurface. When $n=2$, the Newton
polytope $\Delta$ of $f$ does not contain any parallel edges;
otherwise by the duality between $\Delta$  and $\Sur$, the latter
must have a primitive factor. In the general case of
$\Real^{(n)}$, the edges of $\Dl$ that intersect transversally
with a hyperplane cut of $\Dl$ are not all parallel.

Due to  duality between tropical hypersurfaces and their Newton
polytope, the point \symmetry\ of a primitively irreducible
tropical  starred  hypersurface  also induces a \symmetry\ of the
corresponding Newton polytope. Moreover, in the case of Newton
polytopes, the \symmetry \ is always a point \symmetry.


\begin{theorem}\label{thm:symetryOfPolytope} Let $\Dl$ be the Newton polytope of
a primitively irreducible polynomial  $f \in \tTrop[\lm_1, \dots,
\lm_n]$, assume $\Dl$ has empty subdivision, and let $\trn{\Dl}$
be the Newton polytope of $\trn{f}$. Then $\trn{\Dl}$ is a point
\symmetry\ of  $\Dl$, and vise versa.
\end{theorem}

Note that the point-symmetry need not be along a lattice point,
and might be along any point $\bfo \in \Real^{(n)}$.

\begin{proof} We prove the theorem for the boundary
$\partial \Dl$, which is enough  since $\Dl$ is convex. $f$ and
$\trn{f}$, in $\tTrop[\lm_1, \dots,\lm_n ]$,  determine the Newton
polytopes $\Dl$, and $\trn{\Dl}$ uniquely. $\Sur = Z(f)$ is dual
to $\Dl$ and $\trn{\Sur}$ is dual to $\trn{\Dl}$.   The \symmetry\
from $\Sur$ to $\trn{\Sur}$, cf. Theorem \ref{thm:symmOfSur},
completes the proof.
\end{proof}

\begin{corollary}
$\trn{(\trn{\Dl})} = \Dl$ and thus the relation in Theorem
\ref{thm:symetryOfPolytope} induces a duality of lattice
polytopes.
\end{corollary}
\begin{remark} Theorem \ref{thm:symetryOfPolytope} can be extended
to polytopes of the same type on $\Int^{(n)}$. In this case the
assigned  polynomials are Laurent polynomials, i.e. tropical
polynomials over $\tTrop[\lm_1,\lm_1^{-1}, \dots, \lm_n,\lm_n^{-1}
]$, whose Newton polytopes are lattice polytopes over $\Lt =
\Int^{(n)}$.
\end{remark}
\begin{lemma}
 Any integral translation of $\sym{\Dl}$ on $\Lt$ is
also a \symmetry\ of $\Dl$.
\end{lemma}
\begin{proof} Assume $\Dl$ is the Newton polytope of $f\in \tTrop[\lm_1, \dots,\lm_n ]$,
 and take $Z(\trn{f})$ which determines the Newton polytope $\trn{\Dl}$
uniquely up to integral translation on $\Lt$; that is, for any
monomial $h \in \tTrop[\lm_1, \dots,\lm_n ]$, $Z(\trn{f})= Z(h
\trn{f})$. But the Newton polytope of $h \trn{f}$ is just an
integral translation of $\trn{\Dl}$ on $\Lt$.
\end{proof}

\begin{example}\label{exm:4}
\begin{figure}
\setlength{\unitlength}{0.8cm}
\begin{picture}(10,8)(0,0)
\thicklines
 \dottedline{0.15}(2,0.2)(5.9,3.9)
 \dottedline{0.15}(1.4,0.8)(6.5,3.4)
 \dottedline{0.15}(2,1.8)(5.9,2.4)
 \dottedline{0.15}(3.35,0.8)(4.55,3.3)


\put(1,3){$\begin{array}{cccccccccccc}
             \circ & \circ & \circ & \bullet & \circ & \circ & \circ & \circ & \circ & \circ & \circ & \circ \\
              \circ & \bullet &\circ & \circ & \bullet & \circ & \circ & \circ & \circ & \circ & \circ & \circ \\
              \circ & \circ &\circ & \circ & \circ & \circ & \circ & \circ & \circ & \circ & \circ & \circ \\
              \circ & \circ &\bullet & \bullet & \circ & \circ & \circ & \circ & \circ & \circ & \circ & \circ \\
              \circ & \circ &\circ & \circ & \circ & \circ & \circ & \bullet & \circ & \circ & \circ & \circ \\
              \circ & \circ &\circ & \circ & \circ & \bullet & \circ & \circ & \bullet & \circ & \circ & \circ \\
              \circ & \circ &\circ & \circ & \circ & \circ & \circ & \circ & \circ & \circ & \circ & \circ \\
              \circ & \circ &\circ & \circ & \circ & \circ & \bullet & \bullet & \circ & \circ & \circ & \circ \\
             \circ & \bullet & \bullet & \circ & \circ & \circ & \circ & \circ & \circ & \circ & \circ & \circ \\
              \circ & \circ &\circ & \circ & \circ & \circ & \circ & \circ & \circ & \circ & \circ & \circ \\
              \bullet & \circ &\circ & \bullet & \circ & \circ & \circ & \circ & \circ & \circ & \circ & \circ \\
              \circ & \bullet &\circ & \circ & \circ & \circ & \circ & \circ & \circ & \circ & \circ & \circ \\
               \end{array} $}
\thicklines
\drawline(2,0.2)(1.4,0.8)(2,1.8)(2.7,1.8)(3.3,0.75)(2,0.2)

\thicklines
\drawline(5.9,3.95)(6.6,3.4)(6,2.35)(5.25,2.35)(4.6,3.4)(5.9,3.95)

\thicklines
\drawline(3.95,5.55)(3.3,4.4)(2.65,4.4)(2,5.5)(3.3,6.05)(3.95,5.55)

\thinlines \drawline[-20](1.4,0.8)(3.95,5.55)
\drawline[-20](2.7,1.8)(2.65,4.4)\drawline[-20](3.3,0.75)(2,5.5)
\drawline[-20](2,0.2)(3.3,6.05)

 \put(0.8 ,0.8){$\Dl$} \put(6.1 ,3.9){$\sym{\Dl}$}
\put(2.8 ,6.2){$\sym{\Dl_t}$}

\end{picture}
\caption{Illustration for Example \ref{exm:4}. \label{fig:4} }
\end{figure} Let $\Dl$ be the lattice polytope in $\Net^{(n)}$
with vertices $(1,0), (0,1), (3,2), (1,3)$, and  $(2,3)$, which
has no parallel edges.  Assign to $\Dl$ the primitively
irreducible polynomial $$f = \lm_1 + \lm_2 +\lm_1^3 \lm_2 +
\lm_1\lm_2^3 +\lm_1^2 \lm_2^3$$ in $\tTrop[\lm_1,\dots, \lm_n]$
(with Newton polytope $\Delta$), and compute  $$\sym{f} =
\lm_1^8\lm_2^8 + \lm_1^9\lm_2^7 + \lm_1^6\lm_2^7 +\lm_1^8\lm_2^5 +
\lm_1^7\lm_2^5,$$ whose Newton polytope $\sym{\Delta}$ has the
vertices $(8,8), (9,7), (6,7), (8,5)$ and $(7,5)$. Then
$\sym{\Delta}$ is a point \symmetry\ of $\Delta$ and vise versa.
See Fig \ref{fig:4}, the dotted lines show the symmetry between
$\Delta$ and $\sym{\Dl}$. The dashed lines show the symmetry
between $\Delta$ and $\sym{\Dl_t}$, an integral translation of
$\sym{\Dl}$ on $\Lt$.
\end{example}


\end{document}